\documentclass[11pt,bezier]{article}
\usepackage{amsmath,amssymb,amsfonts,euscript,graphicx}

\newtheorem{preproof}{{\bf \indent Proof.}}

\newenvironment{proof}[1]{\begin{preproof}{\rm
               #1}\hfill{$\Box$}}{\end{preproof}}

\newtheorem{prop}{\bf\indent Proposition}[section]
\newtheorem{defn}[prop]{\bf\indent Definition}
\newtheorem{cor}[prop]{\bf\indent Corollary}
\newtheorem{example}[prop]{\bf\indent Example}
\newtheorem{thm}[prop]{\bf\indent Theorem}

\newtheorem{remark}[prop]{\bf\indent Remark}
\newtheorem{lem}[prop]{\bf\indent Lemma}
\title{\bf \Large  Integral closures, Primary Hyperideals and Hypervaluation Hyperideals of Kranser Hyperrings }
\author{{\normalsize  {\sc M. J. Nikmehr${}^{\mathsf{a}}$, {\sc R. Nikandish${}^{\mathsf{b}}$} and {\sc A. Yassine${}^{\mathsf{a}}$}  }
}\vspace{3mm}\\
{\footnotesize{${}^{\mathsf{a}}$\it Faculty of Mathematics, K.N. Toosi
University of Technology, }}\\
{\footnotesize{\rm P.O. BOX \rm{16315-1618}, Tehran, Iran}}\\
{\footnotesize{ $\mathsf{nikmehr@kntu.ac.ir}$}}\quad\quad
{\footnotesize{$\mathsf{yassine\_ali@email.kntu.ac.ir}$}}\\
{\footnotesize{${}^{\mathsf{b}}$\it Department of Mathematics, Jundi-Shapur University of Technology,}}\\
{\footnotesize{\rm P.O. BOX \rm{64615-334},
Dezful, Iran}}\\
{\footnotesize{ $\mathsf{r.nikandish@ipm.ir}$}}\\
{\footnotesize{$\mathsf{}$ }}}
\date{}

\begin{document}

\maketitle
\begin{abstract}
{In this paper, the notions of integral closure of hyperrings and hyperideals in a Krasner hyperring $(R, +, \cdot)$ are defined and some basics properties of them are studied. We define also the notion of hypervaluation hyperideals and then a relations between hypervaluations, integral closure of hyperideals and primary hyperideals are studied. In fact it is shown that the integral closure of a hyperideal is determined by the hypervaluation Krasner hyperrings.}
\end{abstract}
\begin{center}{\section{Introduction
}}\end{center}
\par Algebraic hyperstructure theory is a well generalization of classical algebraic theory and it was introduced by Marty in 1934 \cite{Marty} at the 8th Congress of Scandinavian Mathematicians.  In a classical algebraic structure, the composition of two elements is an element, while in an algebraic hyperstructure, the composition of two elements is a set. Since then, several mathematicians studied Algebraic hyperstructures, see \cite{Corsini,Leoreanu,Davvaz,Davvaz3,Davvaz2,Vougiouklis}. Marty introduced the concept of hyperoperations on  hypergroups which are generalization of groups, and this led to many concepts of hyperrings. Hyperrings are hyperstructures together with two binary (hyper)operations, addition and multiplication, in which addition or multiplication is a hyperoperation. For examples, $(R,+,\cdot )$ is a general hyperring when the addition and the multiplication are two hyperoperations such that $(R,+)$ is a hypergroup and $(R,\cdot )$ is a semihypergroup and the multiplication is distributive over the addition, we have also multiplicative hyperrings. Another type of hyperring is the additive hyperring, if only the addition is a hyperoperation and the multiplication is a usual operation. One of the additive hyperrings is the Krasner hyperring which was introduced by Krasner in 1983 \cite{Krasner}. We refer the reader to \cite{isomorphism,Davvaz,Salasi,DV} for the notions of hyperrings. In this paper, we will take hyperring to mean Krasner hyperring unless otherwise specified.

\par Integral closure play a central role in algebraic geometry and number theory, it is a formulation of ideals began in the 1930's with the work of Krull and Zariski. The integral closure of ideals arise in commutative algebra upon understanding the growth of ideals, and there are many approaches lead to the concept of integral closure of ideals. One of these approaches is to finde an element $c$ such that for all large $n$, $cr^n \in I^n$ (for more on integral closure see \cite{Atiyah,Swanson}).

 Hypervaluation on a hyperring was introduced by Davvaz and Salasi in \cite{Salasi} which is a very useful tool in hyperring theory. They stated and proved some properties relating to these concepts. In Section 3 of this article, at first we define the notions of the integral closure of hyperrings and hyperideals and by considering these notions we obtain some results. In Section 4, we define hypervaluation hyperideals and we prove some interesting results concerning this concept. We show that the integral closure of a hyperideal is determined by the hypervaluation hyperrings, and we present the relationship between hypervaluations and primary hyperideals.

\begin{center}{\section{Preliminaries
}}\end{center}
\par In this section, we give the definition of all notations that will be used in this paper.
\begin{defn} $(${\rm\cite[Definition 3.1.1]{Davvaz}}$)$
A Krasner hyperring is an algebraic hyperstructure $(R,+,\cdot )$ which satisfies the following axioms:

{\bf(1)} $(R, +)$ is a canonical hypergroup, i.e.,
\begin{enumerate}

\item for every $x,y,z \in R$, $x+(y+z) = (x+y)+z$,

\item for every $x,y \in R$, $x+y = y+x$,

\item there exists $x\in R$ such that $0+x=\{x\}$ for every $x \in R$,

\item for every $x \in R$ there exists a unit element $x' \in R$ such that $0\in x+x'$; $x'$ is the opposite of $x$ and we shall write $-x$ for $x'$,

\item $z\in x+y$ implies $y\in -x+z$ and $x\in z-y$.
\end{enumerate}

{\bf(2)} $(R,\cdot )$ is a semigroup having zero as a bilaterally absorbing element, i.e., $x \cdot 0 = 0 \cdot x = 0$.

{\bf(3)} The multiplication is distributive with respect to the hyperoperation $+$.
\end{defn}

We call $0$ the zero of the Krasner hyperring $(R,+,\cdot )$. For $x \in R$, let $-x$ denote the unique inverse of $x$ in $(R, +)$. Then $-(-x) = x$, for all $x \in R$. In addition, $(x+y)\cdot (z+w) \subseteq x\cdot z+x\cdot w+y\cdot z+y\cdot w$, $(-x)\cdot y = x \cdot (-y) = -(x\cdot y)$, for every $x,y,z,w \in R$. A Krasner hyperring $R$ is called commutative (with unit element) if $(R,\cdot )$ is a commutative semigroup (with unit element).

In this paper all hyperrings are commutative with unit element.
\begin{defn} $(${\rm\cite[Definition 3.1.2]{Davvaz}}$)$
{\bf(1)} $(R,+,\cdot )$ is called a Krasner hyperfield if $(R,+,\cdot )$ is a Krasner hyperring and $(R,\cdot )$ is a group.

{\bf(2)} A Krasner hyperring $(R,+,\cdot )$ is called a hyperdomain if $R$ is a commutative hyperring with unit element and $ab=0$ implies that $a=0$ or $b=0$ for all $a,b\in R$.
\end{defn}
\begin{defn} $(${\rm\cite{Davvaz}}$)$
{\bf(1)} A nonempty set $I$ of a Krasner hyperring $R$ is called a hyperideal if for all $a,b \in I$ and $r\in R$ we have $a-b \subseteq I$ and $a\cdot r \in I$.

{\bf(2)} A proper hyperideal $P$ of a Krasner hyperring $R$ is called a prime hyperideal if whenever $ab \in P$, either $a \in P$ or $b \in P$.

{\bf(3)} Let $R$ and $S$ be Krasner hyperrings. A mapping $f: R \rightarrow S$ is called a homomorphism if for all $a,b \in R$, $f(a + b) \subseteq f(a) + f(b)$ and $f(a \cdot b) = f(a) \cdot f(b)$.
\end{defn}

\begin{defn} $(${\rm\cite{Ramaruban}}$)$
{\bf(1)} Let $R$ be a Krasner hyperring. An $R$-hypermodule $M$ is a commutative hypergroup with respect to addition, together with a map $R \times M \rightarrow M$, given by $(r,m) \rightarrow r \cdot m = rm \in M$, such that for all $a,b \in R$ and $m_1, m_2 \in M$ we have:

\begin{enumerate}

\item $(a+b)m_1 = am_1+bm_1$,

\item $a(m_1+m_2) = am_1+am_2$,

\item  $(ab)m_1 = a(bm_1)$,

\item $a0_M = 0_Rm_1 = 0_M$, where $0_M, 0_R$ are the zero elements of $M$ and $R$, respectively,

\item  $1m_1 = m_1$, where $1$ is the multiplicative identity in $R$.
\end{enumerate}

{\bf(2)} A $R$-hypermodule $M$ is called finitely generated, if there exists a finite subset $\{x_1, x_2, \ldots , x_n\}$ of $M$ such that $M = \{z : \exists$ $r_1, \ldots , r_n \in R, n \in \mathbb{N}$ such that $z \in \sum_{i=1}^nr_ix_i\}$. The set $\{x_1, x_2, \ldots , x_n\}$ is called the generating set.
\end{defn}

\begin{center}{\section{Integral closure of Krasner hyperrings and hyperideals
}}\end{center}
\par The principal purpose of this section is to generalize the concept of integral closure of rings and ideals for Krasner hyperrings and hyperideals and we study some of their properties. Before presenting the notion of integral closedness for Krasner hyperrings and hyperideals, we remind some necessary concepts. Let $R$ be a subring of a ring $R'$. An element $r$ of $R'$ is said to be integral over $R$ if there exist an integer $n$ and elements $a_i \in R$, $i = 1, \ldots ,n$, such that $r^n + a_1r^{n-1} + a_2r^{n-2} + \cdots + a_{n-1}r + a_n = 0$. If every element of $R'$ is integral over $R$ we say that $R'$ is integral over $R$. If the elements of $R$ are the only elements of $R'$ which are integral over $R$, we say that $R$ is integrally closed in $R'$. If $R$ is integrally closed in its total quotient ring, we say simply that $R$ is integrally closed (see \cite[Definition 4.2]{Larsen}). These definitions allow us to introduce the notion of integrally closed hyperrings and hyperideals, that are equivalent to the notion of integral closedness for rings in the case in which $1 + (-1) = \{0\}$ (see \cite[Proposition 3.1]{Ramaruban}). Now, we give the following definitions:

\begin{defn}\label{2.1}
 Let $R$ be a Krasner hyperring and $I$ a hyperideal of $R$ and let $r \in R$.

{\bf(i)} We say that $r$ is integral over $I$ precisely when there are $n$ elements $a_i \in I^i$ $(1 \leq i \leq n)$ such that  $0 \in r^n + a_1r^{n-1} + a_2r^{n-2} + \cdots + a_{n-1}r + a_n$. The set $\overline{I}$ that contains all elements of $R$ which are integral over $I$ is the integral closure of $I$. The hyperideal $I$ is said to be integrally closed if $I = \overline{I}$ and for every hyperideal $J$ of $R$, if $J \subseteq \overline{I}$, then $J$ is called integral over $I$.

 {\bf(ii)} Let $Q(R)$ be the quotient hyperfield $($hyperfield of fractions$)$ of $R$ and $x \in Q(R)$. We say that $x$ is integral over $R$ precisely when there are elements $a_1, \ldots ,a_n \in R$ $(1 \leq n)$ such that $0 \in x^n + a_1x^{n-1} + a_2x^{n-2} + \cdots + a_{n-1}x + a_n$. The set $\overline{R}$ that contains all elements of $Q(R)$ which are integral over $R$ is the integral closure of $R$. The Krasner hyperring $R$ is said to be integrally closed if $R = \overline{R}$.
\end{defn}

\par Before studying the properties of the integral closedness of Krasner hyperrings and hyperideals, we give the following example.

\begin{example} \label{2.2}{\rm
Let $R = \{0, a, b, c\}$ be a Krasner hyperring with the hyperaddition $+$ and the multiplication $\cdot$ defined as follows:
\begin{center}{
\begin{tabular}{ c|c c c c}
 $+$ & $0$ & $a$ & $b$ & $c$\\
 \hline
 $0$ & $0$ & $a$ & $b$ & $c$ \\
 $a$ & $a$ & \{$0$, $b$\} & \{$a$, $c$\} & $b$ \\
 $b$ & $b$ & \{$a$, $c$\} & \{$0$, $b$\} & $a$ \\
 $c$ & $c$ & $b$          &      $a$       & $0$\\
\end{tabular}
\hspace{1cm}
\begin{tabular}{ c|c c c c}
 $\cdot$ & $0$ & $a$ & $b$ & $c$\\
 \hline
 $0$ & $0$ & $0$ & $0$ & $0$ \\
 $a$ & $0$ & $a$ & $b$ & $c$ \\
 $b$ & $0$ & $b$ & $b$ & $0$ \\
 $c$ & $0$ & $c$  &  $0$   & $c$\\
\end{tabular}}
\end{center}
It is easily seen that $A = \{0\}$, $B = \{0, b\}$, $I = \{0, c\}$ and $J = \{0, b, c\}$ are hyperideals of $R$. We note that $0 \in \overline{A}$ but $a,b,c \notin \overline{A}$. Hence $\overline{A} = \{0\}$, and so $A$ is an integrally closed hyperideal of $R$. In the hyperideal $B$,  $a \notin \overline{B}$, since $0 \notin a^n + a_1a^{n-1} + a_2a^{n-2} + \cdots + a_{n-1}a + a_n = \{a, c\}$, for every $a_i \in B^i$ $(1 \leq i \leq n)$. Also $c \notin \overline{B}$, as $0 \notin c^n + a_1c^{n-1} + a_2c^{n-2} + \cdots + a_{n-1}c + a_n = \{c\}$,  for every $a_i \in B^i$ $(1 \leq i \leq n)$. Thus $\overline{B} = \{0, b\}$ and  $B$ is an integrally closed hyperideal of $R$. One may see that $I$ is integrally closed, but $J$ is not integrally closed, since $0 \in a^2 + ba + c = \{0, b\}$.}
\end{example}

\begin{remark}\label{remark}
 Let $(R, +, \cdot)$ be a Krasner hyperring and $I,J$ hyperideals of $R$. Then

{\bf(1)} The integral closure of the hyperideal $R$ in the hyperring $R$ is always $R$, but the integral closure of the hyperring $R$ may be larger than $R$.

{\bf(2)} If $I$ is prime, then $I$ is integrally closed, since for every $r \in \overline{I}$, there are $n$ elements $a_i \in I^i$ $(1 \leq i \leq n)$ such that $0 \in r^n + a_1r^{n-1} + a_2r^{n-2} + \cdots + a_{n-1}r + a_n$, and so $r^n \in (-a_1)r^{n-1} + (-a_2)r^{n-2} + \cdots + (-a_{n-1})r + (-a_n) \subseteq I$. But $I$ is prime, then $r \in I$.

{\bf(3)} Since $0 \in x - x$ for every $x \in I$, we get $I \subseteq \overline{I}$.

{\bf(4)} Recall from {\rm \cite[Definition 4.4]{Ramaruban}} that the radical of $I$, denoted by $\sqrt{I} = \{x \in R : x^n \in I$ for some  positive integer $n\}$, is a hyperideal of $R$ containing $I$. But for every $r \in \overline{I}$, there are $n$ elements $a_i \in I^i$ $(1 \leq i \leq n)$ such that $0 \in r^n + a_1r^{n-1} + a_2r^{n-2} + \cdots + a_{n-1}r + a_n$, and so $r^n \in (-a_1)r^{n-1} + (-a_2)r^{n-2} + \cdots + (-a_{n-1})r + (-a_n) \subseteq I$. This means that $r \in \sqrt{I}$. Thus $\overline{I} \subseteq \sqrt{I}$.

{\bf(5)} If $I \subseteq J$, then the integral closure of the hyperideal $I$ is contained in the integral closure of the hyperideal $J$.

{\bf(6)} We take the $nilradical$ of $R$ as in {\rm \cite[Definition 3.14]{Ramaruban}}. Since $(R, \cdot)$ is a semigroup having zero as a bilaterally absorbing element, for every $r$ belongs to $nilradical$ of $R$ we have $0 \in r^n + 0 = \{0\}$ for some positive integer $n$ such that $r^n = 0$, and so $r \in \overline{I}$ for every hyperideal $I$.

{\bf(7)} The intersection of any non-empty family of integrally closed hyperideals of $R$ is again an integrally closed.

{\bf(8)} Suppose that $f : R \rightarrow L$ is a hyperring homomorphism. Then $f(\overline{I}) \subseteq \overline{f(I)L}$ since if $r \in \overline{I}$, then there are $n$ elements $a_i \in I^i$ $(1 \leq i \leq n)$ such that $0 \in r^n + a_1r^{n-1} + a_2r^{n-2} + \cdots + a_{n-1}r + a_n$, and so $0 = f(0) \in f(r)^n + f(a_1)f(r)^{n-1} + f(a_2)f(r)^{n-2} + \cdots + f(a_{n-1})f(r) + f(a_n)$. It follows that $f(r) \in \overline{f(I)L}$.
\end{remark}

 In \cite{Musavi}, Davvaz and Musavi investigated some properties of the polynomial hyperrings. Let $x$ be an indeterminate and $R$ a Krasner hyperring. Then $R[x]$ is called the hyperring of polynomials of $x$ over $R$, and by \cite[Theorem 3.2]{Musavi}, one may see that $(R[x], +, \cdot)$ with the two hyperoperations defined
as follows:
$$f(x) + g(x) = \{\sum_{i=0}^{max\{m,n\}}c_ix^i : c_i \in a_i + b_i\} \hspace{1cm} f(x) \cdot g(x) = \{\sum_{k=0}^{m+n}c_kx^k : c_k \in \sum_{i+j=k}a_ib_j\}$$
 is an additive-multiplication hyperring for some $f(x) = \sum_{i=0}^na_ix^i$, $g(x) = \sum_{k=0}^mb_kx^k \in R[x]$.  In the following proposition, we find a relation between  polynomial hyperrings and integral elements.
\begin{prop}\label{2.4}
Let $S$ be a Krasner hyperring and $R$ a subhyperring of $S$ and let $x \in S$. Then the following statements hold:

{\bf(1)} If $x$ is integral over $R$, then $R[x]$ is a finitely generated $R$-hypermodule.

{\bf(2)} If $P$ is a prime hyperideal of $R$ such that $R$ is Krasner hyperdomain and $S = Q(R)$ is the quotient hyperfield of $R$, then $PR[a] \neq R[a]$ or $PR[a^{-1}] \neq R[a^{-1}]$, for every unit element $a$ of $Q(R)$.
\end{prop}
\begin{proof}
{{\bf(1)} Suppose that $x$ is integral over $R$ and $f(x) \in R[x]$ is a polynomial of degree $m$. Then there are elements $b_1, \ldots , b_n \in R$ $(1 \leq n)$ such that $0 \in x^n + b_1x^{n-1} + b_2x^{n-2} + \cdots + b_{n-1}x + b_n$ and there exist $a_0, \ldots , a_m \in R$ such that $f(x) = a_0 + a_1x + \cdots + a_mx^m$, and so $f(x) \in a_0 + a_1x + \cdots + a_mx^{m-n}((-b_1x^{n-1}) + (-b_2x^{n-2}) + \cdots + (-b_{n-1})x + (-b_n))$. Therefore $f(x) \in a_0' + a_1'x + \cdots + a_{m-1}'x^{m-1}$. By repeating this argument $m-n$ times, we have $f(x) \in a_0'' + a_1''x + \cdots + a_n''x^n$. Hence $R[x]$ is  generated by $\{1, x, \ldots , x^n\}$ as an $R$-hypermodule.

{\bf(2)} There is no loss of generality in assuming that $R = R_P$, and so $R$ is a local hyperdomain and $P$ is maximal in $R$. If $PR[a^{-1}] = R[a^{-1}]$, then $1 \in \sum_{i=0}^na_ia^{-i}$, where $a_i$'s $\in P$. It follows that $0 \in 1 - \sum_{i=0}^na_ia^{-i}$, and so $0 \in 1 - a_0 - a_1a^{-1} - \cdots - a_na^{-n}$. Note that this means that there exists $z\in 1 - a_0$ such that $0 \in z - a_1a^{-1} - \cdots - a_na^{-n}$. By \cite[Proposition 3.15]{Ramaruban}, $z$ is unit in $R$. Hence $0 \in a^{n} - z^{-1}a_1a^{n-1} - \cdots - z^{-1}a_n$, and thus $a \in \overline{R}$. Therefore $R[a] \subseteq \overline{R}$. Now, let $M$ be a maximal ideal in $R[a]$. We show that $M \cap R$ is a maximal hyperideal of $R$. By the choice of $M$, $R[a]/M$ is a hyperfield. Suppose that $x \in R/(M\cap R)$. Then $x^{-1}\in R[a]/M$, and so $x^{-1}\in \overline{R/(M\cap R)}$, since $R[a] \subseteq \overline{R}$. Therefore, there are elements $b_1, \ldots , b_n \in R/(M\cap R)$ $(1 \leq n)$ such that $0 \in x^{-n} + b_1x^{1-n} + b_2x^{2-n} + \cdots + b_{n-1}x^{-1} + b_n$. It follows that $0 \in x^{n-1}(x^{-n} + b_1x^{1-n} + b_2x^{2-n} + \cdots + b_{n-1}x^{-1} + b_n) = x^{-1} + b_1 + b_2x + \cdots + b_{n-1}x^{n-2} + b_nx^{n-1}$. Hence $x^{-1} \in -b_1 - b_2x - \cdots - b_{n-1}x^{n-2} - b_nx^{n-1} \subseteq R/(M\cap R)$. This means that $R/(M\cap R)$ is a hyperfield, and hence $M \cap R$ is a maximal hyperideal of $R$, by \cite[Proposition 3.7]{Harijani}. Since $R$ is local, $M \cap R = P$, and thus $PR[a] \neq R[a]$.}
\end{proof}
\par In the sequel, we generalize some classical results in integrally closed ring theory. In fact, we use hyperideals and hypermodules to express integral closure. First, we give the following lemma.

\begin{lem}\label{2.9}
Let $(R, +, \cdot)$ be a Krasner hyperring, $I$ a hyperideal of $R$ and let $r \in R$. Then $r \in \overline{I}$ if and only if $(I + (r))^n = I(I + (r))^{n-1}$ for some integer $n$.
\end{lem}
\begin{proof}
{$\Rightarrow$) Suppose that $r \in \overline{I}$. Then there are $n$ elements $a_i \in I^i$ $(1 \leq i \leq n)$ such that $0 \in r^n + a_1r^{n-1} + a_2r^{n-2} + \cdots + a_{n-1}r + a_n$, and so $r^n \in (-a_1)r^{n-1} + (-a_2)r^{n-2} + \cdots + (-a_{n-1})r + (-a_n) \subseteq I(I+(r))^{n-1}$. This means that $(I+(r))^n \subseteq I(I + (r))^{n-1}$, and hence $(I+(r))^n = I(I + (r))^{n-1}$.

$\Leftarrow$) Suppose that $(I+(r))^n = I(I + (r))^{n-1}$, for some positive integer $n$. Then $r^n \in I(I+(r))^{n-1}$, and so there are $n$ elements $c_1, \ldots , c_n \in I$ such that $c_i \in I^i$ for every $i$ and $r^n \in c_1r^{n-1} + c_2r^{n-2} + \cdots + c_{n-1}r + c_n$. Hence $r \in \overline{I}$.
}
\end{proof}

\begin{thm} \label{2.10}
Let $(R, +, \cdot)$ be a Krasner hyperring, $I$ a hyperideal of $R$ and let $r \in R$. If $r \in \overline{I}$, then there exists a finitely generated $R$-hypermodule $M$ such that $rM \subseteq IM$ and $0 \in r^kx$ $($for some $k \in \mathbb{N})$ for every $x \subseteq R$ such that $xM = 0$, where $x$ is a set of elements of $R$. In fact, the converse is true whenever the set $\{y : y \in \sum_{i=1}^ny_i - y_i$ for some $y_i \in I\} = 0$.

\end{thm}
\begin{proof}
{Suppose that $r \in \overline{I}$. Then there are $n$ elements $a_i \in I^i$ $(1 \leq i \leq n)$ such that $0 \in r^n + a_1r^{n-1} + a_2r^{n-2} + \cdots + a_{n-1}r + a_n$. Let $J$ be the hyperideal generated by $\{a_1, a_2, \ldots , a_n\}$. Then $I \subseteq J$ and $r \in \overline{J}$. By Lemma \ref{2.9}, $(J + (r))^n = J(J + (r))^{n-1}$, for some integer $n$. Set $M =: (J + (r))^{n-1}$. Then $M$ is a finitely generated $R$-hypermodule and $rM = r(J + (r))^{n-1} \subseteq JM \subseteq IM$. Also, suppose that $x \in R$ such that $xM = 0$. Then $xr^{n-1} \in xM = 0$, and hence $0 \in r^{n-1}x $. Conversely, suppose that there exists a finitely generated $R$-hypermodule $M$ such that $rM \subseteq IM$ and $0 \in r^kx $ $(k \in \mathbb{N})$, for every $x \in R$ such that $xM = 0$. Then there exists a finite subset $\{m_1, m_2, \ldots , m_n\}$ of $M$ such that $M = \{z$ : there exist $r_1, \ldots , r_n \in R$ such that $z \in \sum_{i=1}^nr_im_i\}$ and $rM \subseteq IM$. Hence $rm_i \in \sum_{j=1}^na_{ij}m_j$, for every $i \in \{1, \ldots , n\}$ and some $a_{ij} \in I$. Let $A = (\delta_{ij}r - a_{ij})$ be the matrix in which $\delta_{ij} = 1$ whenever $i=j$, or else $\delta_{ij} = 0$ and suppose that $m$ is the vector $(m_1, \ldots , m_n)^T$. It is easy to see that $rm_i - a_{i1}m_1 - a_{i2}m_2 - \cdots - a_{in}m_n \subseteq a_{i1}m_1 - a_{i1}m_1 + a_{i2}m_2 - a_{i2}m_2 + \cdots + a_{in}m_n - a_{in}m_n = 0$ for every $i \in \{1, \ldots , n\}$. It follows that $0 = adj(A)Am = det(A)m$, and so $det(A)m_i = 0$ for every $i \in \{1, \ldots , n\}$. Therefore, $0 = det(A)M$, and hence by the hypothesis $0 \in r^kdet(A)$ for $k \in \mathbb{N}$. Thus $r \in \overline{I}$.
}
\end{proof}
\par We close this section with the following corollary.
\begin{cor}\label{2.7}
Let $I$ be a hyperideal of the Krasner hyperring $R$ such that $\{y : y \in \sum_{i=1}^ny_i - y_i$ for some $y_i \in I\} = 0$. Then $\overline{I}$ is a hyperideal of $R$.
\end{cor}
\begin{proof}
{Let $I$ be a hyperideal of the Krasner hyperring $R$. It is easy to see that $\alpha r \in \overline{I}$, for every $\alpha \in R$ and $r \in \overline{I}$. It is sufficient to show that $x+y \subseteq \overline{I}$ for every $x,y \in \overline{I}$. Let $x,y \in \overline{I}$. Then there are elements $a_i \in I^i$ $(1 \leq i \leq n)$ and $b_j \in I^j$ $(1 \leq j \leq m)$ such that $0 \in x^n + a_1x^{n-1} + a_2x^{n-2} + \cdots + a_{n-1}x + a_n$ and $0 \in y^m + b_1y^{m-1} + b_2y^{m-2} + \cdots + b_{m-1}y + b_m$. Let $I'$ be the hyperideal generated by $\{a_1, \ldots , a_n, b_1, \ldots , b_m\}$. Then $x,y \in \overline{I'}$. Let $J = I' + (x)$ and $K = I' + (x,y) = J + (y)$. Since $x,y \in \overline{I'}$, it follows from Lemma \ref{2.9} that $J^{n+1} = (I' + (x))^{n+1} = I'(I' + (x))^n = I'J^n$ for some integer $n$ and $K^{m+1} = (J + (y))^{m+1} = J(J + (y))^m = JK^m$ for some integer $m$. Hence $K^{m+n+1} = JK^{m+n} = \cdots = J^{n+1}K^m = I'J^{n}K^m \subseteq I'K^{m+n} \subseteq K^{m+n+1}$. Thus $I'K^{m+n} = K^{m+n+1}$, and hence $I'(I' + (x,y))^{m+n} = (I' + (x,y))^{m+n+1}$. As $x+y \subseteq K$, $(x+y)K^{m+n} \subseteq K^{m+n+1} = I'K^{m+n}$. If $aK^{m+n} = 0$ for some $a\in R$, then $a(x+y)^{m+n} = 0$, and so $x+y \subseteq \sqrt{(0:a)}$. Hence, by Theorem \ref{2.10}, $x+y \subseteq \overline{I'}$, and thus $\overline{I}$ is hyperideal.}
\end{proof}

\begin{center}{\section{ Hypervaluation Hyperideals and Primary Hyperideals}}
\end{center}
\par In this section some important results in valuation ring theory are generalized. We define the notion of hypervaluation hyperideals and then a relation between hypervaluations and primary hyperideals is studied.

Let $Q(R)$ be the quotient hyperfield of the Krasner hyperring $R$. Recall from \cite[Definition 4.3]{Harijani} that $R$ is called hypervaluation hyperring if for any $a \in Q(R)$ we have $a \in R$ or $a^{-1} \in R$, and we can easily see that if $R$ is a hypervaluation hyperring, then the set of all hyperideals of $R$ is totally ordered by inclusion. The concept of hypervaluation on a hyperfield was introduced in \cite{Harijani}. A hypervaluation on $Q(R)$ is a surjective map $\nu : Q(R) \rightarrow G_{\infty}$, where $G$ is a totally ordered Abelian group with an element $\infty$, such that

(a) $\nu(x) = \infty$ if and only if $x = 0$;

(b) $\nu(-x) = \nu(x)$;

(c) $\nu(x\cdot y) = \nu(x) \cdot \nu(y)$;

(d) If $z \in x + y$, then $\nu(z) \geq min\{\nu(x), \nu(y)\}$.

For more details we refer the reader to \cite{Davvaz,Salasi,Ramaruban}. The non-negative hypervaluations on $Q(R)$ give us a special class of hyperideals. In the following these hyperideals (hypervaluation hyperideals) are introduced.

\begin{defn}
Let $I$ be a hyperideal of the Krasner hyperdomain $R$ such that $Q(R)$ is the quotient hyperfield of $R$. Then $I$ is said to be
hypervaluation hyperideal if there is a hypervaluation hyperring $V$ of $Q(R)$ containing $R$ and a hyperideal $J$ of $V$ such that $I = J \cap R$. We say that $I$ is a $\nu$-hyperideal whenever $\nu$ is the hypervaluation determined by $V$.
\end{defn}

\begin{prop}\label{c}
Let $I$ be a hyperideal of the Krasner hyperdomain $R$ and let $\nu: Q(R) \rightarrow G_{\infty}$ be a hypervaluation on $Q(R)$ and nonnegative on $R$. Then the following statements are equivalent:

{\bf(a)} $I$ is a $\nu$-hyperideal.

{\bf(b)} For $x,y \in R$, if $x \in I$ and $\nu(y) \geq \nu(x)$, then $y \in I$.

{\bf(c)} If $V$ is the hypervaluation hyperring of $\nu$, then $IV\cap R = I$.
\end{prop}
\begin{proof}
{{\bf(a)} $\Rightarrow$ {\bf(b)} Let $I$ be a $\nu$-hyperideal and $x,y \in R$ be such that $x \in I$ and $\nu(y) \geq \nu(x)$. Then there is a hypervaluation hyperring $V$ of $Q(R)$ containing $R$ and a hyperideal $J$ of $V$ such that $I = J \cap R$, and so $y = \frac{y}{x}\cdot x \in Vx\cap R \subseteq J\cap R = I$.

{\bf(b)} $\Rightarrow$ {\bf(c)} Let $y \in VI\cap R$. Then $y \in \sum_{i=1}^na_ib_i$, where $a_i$'s $\in I$ and $b_i$'s $\in V$. Suppose that  $\nu(a_i) = min\{\nu(a_1), \ldots , \nu(a_n)\}$. Then $\nu(y) \geq min\{\nu(a_1)\nu(b_1), \ldots , \nu(a_n)\nu(b_n)\} \geq min\{\nu(a_1), \ldots , \nu(a_n)\} = \nu(a_i)$. Hence $y \in I$ and thus $IV\cap R = I$.

{\bf(b)} $\Rightarrow$ {\bf(c)} is straightforward.
}
\end{proof}
\begin{example}\label{example}{\rm

 Let $I$ be a $\nu$-hyperideal of the Krasner hyperdomain $R$. Then $\sqrt{I}$ is a $\nu$-hyperideal of $R$. Further, $\sqrt{I}$ is a prime hyperideal of $R$. In fact, let $x,y \in R$ be such that $x\in \sqrt{I}$ and $\nu(y) \geq \nu(x)$. Then there exists a positive integer $n$ such that $x^n \in I$. Hence $\nu(y^n) \geq \nu(x^n)$. Since $I$ is a $\nu$-hyperideal, we get $y^n \in I$. Thus $\sqrt{I}$ is a $\nu$-hyperideal. Furthermore, if $ab \in \sqrt{I}$ for some $a,b \in R$, then $a^nb^n \in I$ for some positive integer $n$. Assume that $\nu(a) \geq \nu(b)$, then $\nu(a^{2n}) \geq \nu(a^{n}b^{n})$, which means that $\sqrt{I}$ is a prime hyperideal of $R$.

}
\end{example}
\par We know that every valuation ring is integrally closed \cite[Proposition 5.18]{Atiyah}. In the following we show that this property is also true in hypervaluation Krasner hyperrings.

\begin{prop}\label{def}
Every hypervaluation Krasner hyperring is integrally closed.
\end{prop}
\begin{proof}
{Let $R$ be a hypervaluation Krasner hyperring with the quotient hyperfield $Q(R)$ and let $r \in \overline{R}$. Then there are elements $a_1, \ldots , a_n \in R$ $(1 \leq n)$ such that $0 \in r^n + a_1r^{n-1} + a_2r^{n-2} + \cdots + a_{n-1}r + a_n$. If $r \in R$, then we are done. Suppose that $r \notin R$. Then $r^{-1} \in R$, and so $r^{-s} \in R$ for every $s \in \mathbb{N}$. It follows that
$$0 \in r + a_1 + a_2r^{-1} + \cdots + a_{n-1}r^{2-n} + a_nr^{1-n}.$$
Hence $r \in -(a_1 + a_2r^{-1} + \cdots + a_{n-1}r^{2-n} + a_nr^{1-n}) \subseteq R$, a contradiction. Thus $R$ is an integrally closed Krasner hyperring.
}
\end{proof}
\par In the following relation between hypervaluation Krasner hyperrings and integral closure of hyperideals is given. In fact, we show that the integral closure of a hyperideal is determined by the hypervaluation Krasner hyperrings. First, we show that for every Krasner hyperdomain $R$, there exists a hypervaluation hyperdomain between $R$ and $Q(R)$.
\begin{thm}\label{existance}
Let $R$ be a Krasner hyperdomain with the hyperfield of fractions $Q(R)$. Then for every prime hyperideal $P$ of $R$ there exists a hypervaluation hyperdomain $V$ of $Q(R)$ containing $R$ such that $M\cap R = P$, where $M$ is a maximal hyperideal in $V$.
\end{thm}
\begin{proof}
{With no loss of generality  assume that $R = R_P$, and so $R$ is local hyperdomain and $P$ is maximal in $R$. Let
\begin{center}
$\Theta = \{(S,M_S) : R \subseteq S \subseteq Q(R)$ and $PS \subseteq M_S\},$
\end{center}
where each $(S,M_S)$ is a local hyperdomain with the maximal hyperideal $M_S$. Then $\Theta$ is not empty, as $(R,P)$. One may see that $\Theta$ is partially ordered by $\leq$, where $(S,M_S) \leq (K,M_K)$ if $S \subseteq K$ and $M_SK \subseteq M_K$, and  every non-empty totally ordered subset of $\Theta$ has an upper bound. Then $\Theta$ has at least one maximal element $(V,M_V)$, by Zorn's Lemma. We show that $V$ is a hypervaluation hyperdomain. Let $a$ be an element of $Q(R)$. Then by Proposition \ref{2.4} (2), either $M_VV[a] \neq V[a]$ or $M_VV[a^{-1}] \neq V[a^{-1}]$. Assume that $M_VV[a] \neq V[a]$. Then by \cite[Corollary 3.9]{Harijani}, there exists a maximal hyperideal $M_{V[a]}$ of $V[a]$ containing $M_VV[a]$. Therefore, $(V[a]_{M_{V[a]}},M_{V[a]}V[a]_{M_{V[a]}})$ is an element of $\Theta$ which contains $V$, and so $V[a]_{M_{V[a]}} = V$. It follows that $a\in V[a]_{M_{V[a]}} = V$. Thus, $V$ is a hypervaluation hyperdomain, as $P$ is maximal in $R$ and $(V,M_V)\in \Theta$. Therefore $P = M_V\cap R$.}
\end{proof}
\begin{thm}\label{IV}
Let $I$ be a hyperideal of the Krasner hyperdomain $R$ and let $V(R)$ be the set of all hypervaluation hyperdomains of the hyperfield of fractions $Q(R)$ of $R$ which contains $R$. Then $\overline{I} = \bigcap_{V\in V(R)}IV\cap R$.
\end{thm}
\begin{proof}
{ Suppose that $r \in \overline{I} \subseteq \bigcap_{V\in V(R)}\overline{I}V\cap R$. By Remark part (8) of \ref{remark}, $\overline{I}V\subseteq \overline{IV}$, and so $r \in \bigcap_{V\in V(R)}\overline{IV}\cap R$. It follows that there are elements $a_1, \ldots , a_n \in I^iV$ $(1 \leq n)$ such that $0 \in r^n + a_1r^{n-1} + a_2r^{n-2} + \cdots + a_{n-1}r + a_n$. Since for every $x,y\in V$, we conclude that $\frac{x}{y}\in V$ or $\frac{y}{x}\in V$, this implies that $x\in yV$ or $y\in xV$. Thus $IV= xV$, for some $x\in I$, and hence $r \in \overline{xV}$. Therefore, there are $m$ elements $b_1, \ldots , b_m \in V$ $(1 \leq m)$ such that $0 \in r^m + b_1xr^{m-1} + b_2x^2r^{m-2} + \cdots + b_{m-1}x^{m-1}r + b_mx^{m}$. This means that $0 \in (\frac{r}{x})^m + b_1(\frac{r}{x})^{m-1} + \cdots + b_{m-1}(\frac{r}{x}) + b_m$. But $\frac{r}{x}\in Q(R)$ and every hypervaluation Krasner hyperring is integrally closed, by Proposition \ref{def}. Hence $\frac{r}{x}\in V$, and thus $r\in xV\subseteq IV$. Therefore, $\overline{I} \subseteq \bigcap_{V\in V(R)}IV\cap R$. For the opposite inclusion, let $r\in \bigcap_{V\in V(R)}IV\cap R$. For every hypervaluation hyperdomain $V$ between the polynomial hyperring $R[\frac{I}{r}]$ and $Q(R)$, we have $r\in IV$, and hence $\frac{I}{r}V = V$. Since by Theorem \ref{existance}, there exists a hypervaluation hyperdomain $V$ of $Q(R)$ contains $R[\frac{I}{r}]$, it follows easily that $\frac{I}{r}V\cap R[\frac{I}{r}] = \frac{I}{r}R[\frac{I}{r}] = R[\frac{I}{r}]$. Therefore $1\in \frac{I}{r}R[\frac{I}{r}]$. This implies that $1\in \sum_{i=1}^n\frac{a_i}{r^i}$ where $a_i \in I^i$ for every $i\in \{1, \ldots , n\}$, and thus $0\in r^n -  a_1r^{n-1} - a_2r^{n-2} - \cdots - a_{n-1}r - a_n$. Hence  $r \in \overline{I}$, as desired.
}
\end{proof}
In the light of Theorem \ref{IV}, we state the following corollary.
\begin{cor}
{\bf(a)} The integral closure of a Krasner hyperdomain $R$ is the intersection of all hypervaluation hyperdomains of the hyperfield of fractions $Q(R)$  containing $R$.

{\bf(b)} Every prime hyperideal is a hypervaluation hyperideal.
\end{cor}
\begin{proof}
{{\bf(a)} Suppose that $r\in \bigcap_{V\in V(R)}V$ where $V(R)$ is the set of all hypervaluation hyperdomains of $Q(R)$ containing $R$. Then $r = \frac{x}{y}$ such that $y\neq 0$. Hence $x\in \bigcap_{V\in V(R)}yV$, and so by Theorem \ref{IV} $x\in \overline{(y)}$. This implies that $0 \in x^n + b_1yx^{n-1} + b_2y^2x^{n-2} + \cdots + b_{n-1}y^{n-1}x + b_ny^{n}$ for some $b_1, \ldots , b_n \in R$ $(1 \leq n)$. Thus $0 \in (\frac{x}{y})^n + b_1(\frac{x}{y})^{n-1} + b_2(\frac{x}{y})^{n-2} + \cdots + b_{n-1}\frac{x}{y} + b_n$, and so $r\in \overline{R}$. But $\overline{R} \subseteq \bigcap_{V\in V(R)}V$, so that $\overline{R} = \bigcap_{V\in V(R)}V$.

{\bf(b)} It follows from Theorem \ref{existance}.}
\end{proof}

 \par In the sequel, we generalize some other results in valuation ring theory. Actually, we find a relationship between hypervaluation hyperideals and primary hyperideals. First, recall that a hyperideal $Q$ in the Krasner hyperring $R$ is primary if $Q$ is proper in $R$ and for every $x,y \in R$ such that $xy \in Q$ we have either $x \in Q$ or $y \in \sqrt{Q}$ (see \cite[Definition 8.1]{Ramaruban}), and a hyperideal $I$ in the Krasner hyperring $R$ is said to be normal if for every $r \in R$, we have $r + I - r \subseteq I$ (see \cite[Definition 3.2.1]{Davvaz}).

\begin{prop}\label{3.5}
Let $I$ be a hyperideal of the Krasner hyperdomain $R$ and let $\nu: Q(R) \rightarrow G_{\infty}$ be a hypervaluation on $Q(R)$ and nonnegative on $R$. Then the following statements hold:

{\bf(a)} If $I$ is a $\nu$-hyperideal and $A,B$ are subsets of $R$ such that $AB \subseteq I$, then either $\{a^2 : a \in A\} \subseteq I$ or $\{b^2 : b \in B\} \subseteq I$. Moreover, if every hyperideal of $R$ is a hypervaluation hyperideal, then $R$ is a hypervaluation hyperdomain.

{\bf(b)} If $I_1,\ldots , I_n$ are $\nu$-hyperideals of $R$ such that for every $i \in \{1, \ldots , n\}$ there exists an element $x_i \in R \setminus I_i$, then $\prod_{i=1}^nx_i \notin I_1\cdots I_n$. Moreover, if $I^n \subseteq I_i^n$ for some $i \in \{1, \ldots , n\}$, then $I \subseteq I_i$.

{\bf(c)} If for every $n$, $I^n$ is a $\nu$-hyperideal of $R$, then $\bigcap_{n=1}^{\infty} I^n$ is a prime hyperideal of $R$.

{\bf(d)} If $R$ is a hypervaluation hyperdomain and $P$ is a prime hyperideal of $R$, then the intersection of the $P$-primary hyperideals of $R$ is prime.

{\bf(e)} Suppose that $\Theta$ is the set of all hypervaluation hyperideals of $R$ such that for every $I_1,I_2 \in \Theta$, there exists $I_3 \in \Theta$ with $I_3 \subseteq I_1\cap I_2$. Then $\sqrt{\bigcap_{I_i\in \Theta}I_i}$ is a prime hyperideal of $R$.
\end{prop}
\begin{proof}
{{\bf(a)} Suppose that there is an element $a \in A$ such that $a^2 \notin I$. Since $I$ is a hypervaluation hyperideal, there is a hypervaluation hyperring $V_{\nu}$ of $Q(R)$ containing $R$ such that $I = IV_{\nu} \cap R$, and since $G$ is totally ordered and for any $b \in B$, we have either $\nu(b) \geq \nu(a)$ or $\nu(a) \geq \nu(b)$. But for any $b \in B$ we have $ab \in I$. Then $b^2 \in I$ for any $b \in B$, and thus $\{b^2 : b \in B\} \subseteq I$. The "moreover" statement is clear since for every $\frac{a}{b} \in Q(R)$ we have either $a^2 \in (ab)$ or $b^2 \in (ab)$, and thus either $\frac{a}{b} \in R$ or $\frac{b}{a} \in R$.

{\bf(b)} Suppose that $x_i \in R \setminus I_i$ for every $i \in \{1, \ldots , n\}$. Then $\nu(a_i) > \nu(x_i)$ for every $a_i \in I_i$, and hence $\prod_{i=1}^n\nu(a_i) > \prod_{i=1}^n\nu(x_i)$ for every $a_i \in I_i$ and every $i \in \{1, \ldots , n\}$. Therefore $\prod_{i=1}^nx_i \notin I_1\cdots I_n$. For the ``moreover" statement, suppose that $I \subseteq I_i$. Then there exists $x\in I \setminus I_i$, and so $x^n \notin I_i^n$, a contradiction.

{\bf(c)} Let $a,b \in R$ such that $ab \in \bigcap_{n=1}^{\infty} I^n$. Then $ab \in I^{2n}$ for each $n$, and by  part (b) either $a \in I^{n}$ or $b \in I^{n}$. Hence, $a \in \bigcap_{n=1}^{\infty} I^n$ or $b \in \bigcap_{n=1}^{\infty} I^n$.

{\bf(d)} Without loss of generality, assume that $R = R_P$, and so by \cite[Remark 7.3]{Ramaruban}, $R_P$ is a local hyperring with the only maximal hyperideal $P$ of $R$. Let $B$ be the intersection of all $P$-primary hyperideals of $R$. If $B = P$, we are done. So suppose that $B \subset P$. Then there exists $x \in P \setminus B$, and hence $x \notin Q$ for some $P$-primary hyperideal $Q$. But $R$ is hypervaluation, then $Q \subset (x) \subseteq P$. On the other hand,  $x^n \in Q_{\alpha}$ for some integer $n$ and $P$-primary hyperideal $Q_{\alpha}$, which means that $(x^n) \subseteq Q_{\alpha}$. Since $\sqrt{(x^n)} = P$, $(x^n)$ is a $P$-primary hyperideal of $R$. Therefore, $\bigcap_{n=1}^{\infty}(x)^n = B$. Thus, by part (c), the intersection of the $P$-primary hyperideals of $R$ is prime.

{\bf(e)} Let $ab \in \sqrt{\bigcap_{I_i\in \Theta}I_i}$ for some $a,b \in R$. Then $a^nb^n \in \bigcap_{I_i\in \Theta}I_i$ for some positive integer $n$. Since $I_i$'s are hypervaluation hyperideals, by part (a), either $a^{2n} \in I_i$ or $b^{2n} \in I_i$ for each $i$. If $a^{2n} \notin I_i$ and $b^{2n} \in I_j$ for some $I_i,I_j \in \Theta$, then there exists $I_k \in \Theta$ such that $I_k \subseteq I_i\cap I_j$. It means that $a^{2n},b^{2n} \notin I_k$, a contradiction. Hence, we can assume that $a^{2n} \in \bigcap_{I_i\in \Theta}I_i$. Thus $\sqrt{\bigcap_{I_i\in \Theta}I_i}$ is a prime hyperideal of $R$.
}
\end{proof}
\par Now, we are ready to state our main result of this section. Let $R$ be a Krasner hyperdomain, $\Theta$ be the set of all hypervaluation hyperideals of $R$ and let $\Delta$ be the set of all $P$-primary hyperideals of $R$ where $P$ is a prime hyperideal. We present a relationship between $\Theta$ and $\Delta$.

\begin{thm}
Suppose that $R$ is a local Krasner hyperdomain such that $\Delta \subseteq \Theta$ if and only if $\Delta$ is totally ordered by inclusion and $\bigcap_{I_i\in \Delta}I_i$ is a prime hyperideal of $R$. Then the equivalence is true for any (not necessarily local) Krasner hyperdomain.

\end{thm}
\begin{proof}
{ Let $R_P$ be a local Krasner hyperdomain with the maximal hyperideal $PR_P$ and let $Q_{\alpha}$'s $\in \Delta$. By \cite[Proposition 8.7]{Ramaruban}, the primary hyperideals of $R_P$ can be written in the form $Q_{\alpha}R_P$ and we have $Q_{\alpha} = Q_{\alpha}R_P\cap R$. Hence, $Q_{\alpha}$ is a hypervaluation hyperideal of $R$ whenever $Q_{\alpha}R_P$ is a hypervaluation hyperideal of $R_P$. Conversely, suppose that $Q_{\alpha}$ is a hypervaluation hyperideal of $R$. Then there exists a hypervaluation hyperring $V$ containing $R$ such that $Q_{\alpha}V\cap R = Q_{\alpha}$. It follows that $Q_{\alpha}V_P\cap R_P = Q_{\alpha}R_P$, since if $x \in Q_{\alpha}V_P\cap R_P$, then $x = \frac{t}{m} = \frac{r}{s}$ for some $r \in R$, $t \in Q_{\alpha}V$ and $m,s \in R \setminus P$. Hence, $0 \in (ts - rm)u$ for some $u \in R \setminus P$, and so $tsu = rmu \in Q_{\alpha}V\cap R = Q_{\alpha}$ (see \cite[page 63]{Ramaruban}). Since $m \in R \setminus P$ and $Q_{\alpha}$ is a $P$-primary hyperideal of $R$, we deduce that $ru\in Q_{\alpha}$, and hence $x = \frac{ru}{su} \in Q_{\alpha}R_P$. This means that $Q_{\alpha}R_P$ is a hypervaluation hyperideal. Now, since $Q_{\alpha} = Q_{\alpha}R_P\cap R$, it is easy to see that $\Delta = \{Q_{\alpha}R_P : Q_{\alpha}R_P$ is a $P$-primary hyperideal of $R_P\}$ is totally ordered by inclusion in $R_P$ if and only if $\Delta' = \{Q_{\alpha} : Q_{\alpha}$ is a $P$-primary hyperideal of $R\}$ is totally ordered by inclusion in $R$. It remains to show that $\bigcap_{Q_{\alpha}\in \Delta'}Q_{\alpha}$ is a prime hyperideal of $R$ if and only if $\bigcap_{Q_{\alpha}R_P\in \Delta}Q_{\alpha}R_P$ is a prime hyperideal of $R_P$. Let $x = \frac{a}{s}, y = \frac{b}{t} \in R_P$ be such that $xy \in \bigcap_{Q_{\alpha}R_P\in \Delta}Q_{\alpha}R_P$ and $x \notin \bigcap_{Q_{\alpha}R_P\in \Delta}Q_{\alpha}R_P$. Then there is $Q_{\beta}R_P\in \Delta$ such that $x\notin Q_{\beta}R_P$, and so $a\notin Q_{\beta}$. Since $xy \in \bigcap_{Q_{\alpha}R_P\in \Delta}Q_{\alpha}R_P$,  $ab \in \bigcap_{Q_{\alpha}R_P\in \Delta}Q_{\alpha}R_P\cap R = \bigcap_{Q_{\alpha}\in \Delta'}Q_{\alpha}$ which is prime. Hence $b \in \bigcap_{Q_{\alpha}R_P\in \Delta}Q_{\alpha}R_P\cap R = \bigcap_{Q_{\alpha}\in \Delta'}Q_{\alpha}$, and so the proof is complete.
}
\end{proof}
We close this paper with the following result.
\begin{thm}\label{3.10}
Let $R$ be a local Krasner hyperdomain with the maximal hyperideal $M$ and let $\Delta = \{Q_{\alpha} : Q_{\alpha}$ is a $M$-primary hyperideals of $R\}$ such that every hyperideal of $\Delta$ is normal and let $\Theta$ be the set of all hypervaluation hyperideals of $R$. If $\Delta \subseteq \Theta$, then $\Delta$ is totally ordered by inclusion and $\bigcap_{Q_{\alpha}\in \Delta}Q_{\alpha}$ is a prime hyperideal of $R$.
\end{thm}
\begin{proof}
{Let $Q_{\alpha} \in \Delta$ be an $M$-primary hyperideal of $R$. First, we show that the set of hyperideals of the hyperring $R/Q_{\alpha}$ is totally ordered by inclusion. Let $x,y \in R$. If $x \notin M$ or $y \notin M$, then by \cite[Corollary 3.10]{Ramaruban}, $x$ or $y$ is unit, and we are done. So, suppose that $x,y \in M$. Then, by \cite[Proposition 4.7 (4)]{Ramaruban}, $\sqrt{Q_{\alpha}^2 + (xy)} = \sqrt{M} = M$. It follows that $Q_{\alpha}^2 + (xy) \in \Delta$, and so $Q_{\alpha}^2 + (xy) \in \Theta$. Hence, by Proposition \ref{c}, either $x^2 \in Q_{\alpha}^2 + (xy)$ or $y^2 \in Q_{\alpha}^2 + (xy)$, as $xy \in Q_{\alpha}^2 + (xy)$. Assume that $x^2 \in Q_{\alpha}^2 + (xy)$. There exist $q \in Q_{\alpha}^2$ and $r \in R$ such that $x^2 \in q + rxy$. Therefore, $x^2 - rxy \subseteq rxy + q - rxy \subseteq Q$, since $Q$ is normal. One may apply part (b) of Proposition \ref{3.5} to see that either $x \in Q_{\alpha}$ or $x - ry \subseteq Q_{\alpha}$. If $x \in Q_{\alpha}$, then $Q_{\alpha} + (x) = Q_{\alpha} \subseteq Q_{\alpha} + (y)$. If $x - ry \subseteq Q_{\alpha}$, then $x \in ry + x - ry$ since $0 \in ry - ry$, and again $Q_{\alpha} + (x) \subseteq Q_{\alpha} + (y)$, and so the set of hyperideals of the hyperring $R/Q_{\alpha}$ is totally ordered by inclusion. Now, let $Q_{\beta}, Q_{\gamma} \in \Delta$. Since $Q_{\beta}\cap Q_{\gamma} \in \Delta$, we have either $Q_{\beta}/(Q_{\beta}\cap Q_{\gamma}) \subseteq Q_{\gamma}/(Q_{\beta}\cap Q_{\gamma})$ or $Q_{\gamma}/(Q_{\beta}\cap Q_{\gamma}) \subseteq Q_{\beta}/(Q_{\beta}\cap Q_{\gamma})$, and so either $Q_{\gamma} \subseteq Q_{\beta}$ or $Q_{\beta}\subseteq Q_{\gamma}$. Thus $\Delta$ is totally ordered by inclusion. It remains to show that $\bigcap_{Q_{\alpha}\in \Delta}Q_{\alpha}$ is a prime hyperideal of $R$. Let $Q_{\alpha}\in \Delta$. Since $\Delta \subseteq \Theta$, there exists a hypervaluation hyperring $V$ containing $R$ such that $Q_{\alpha}V\cap R = Q_{\alpha}$. Suppose that $Q_V = Q_{\alpha}V$. Then $P_V = \sqrt{Q_V}$ is a prime hyperideal of $V$ (see Example \ref{example}), and $P_V\cap R = P$. Assume that $A$ is the intersection of all $P_V$-primary hyperideals of $V$ such that $B = A \cap R$. By Proposition \ref{3.5} (d), $A$ is a prime hyperideal of $V$, and hence $B = A \cap R$ is a prime hyperideal of $R$. Since $\bigcap_{Q_{\alpha}\in \Delta}Q_{\alpha}$ is a primary hyperideal of $R$ and each prime hyperideal is also primary, we conclude that $\bigcap_{Q_{\alpha}\in \Delta}Q_{\alpha} \subseteq B = A \cap R \subseteq Q_{\alpha}$, and since $B$ is prime, we get $\sqrt{\bigcap_{Q_{\alpha}\in \Delta}Q_{\alpha}} \subseteq B \subseteq Q_{\alpha}$, and hence $\sqrt{\bigcap_{Q_{\alpha}\in \Delta}Q_{\alpha}} \subseteq B \subseteq \bigcap_{Q_{\alpha}\in \Delta}Q_{\alpha}$. Thus $\sqrt{\bigcap_{Q_{\alpha}\in \Delta}Q_{\alpha}} = \bigcap_{Q_{\alpha}\in \Delta}Q_{\alpha}$. Now, let $Q_{\beta}, Q_{\gamma} \in \Delta \subseteq \Theta$. Then $Q_{\gamma} = Q_{\beta}\cap Q_{\gamma} \in \Delta \subseteq \Theta$. It follows from Proposition \ref{3.5} (e) that $\sqrt{\bigcap_{Q_{\alpha}\in \Delta}Q_{\alpha}} = \bigcap_{Q_{\alpha}\in \Delta}Q_{\alpha}$ is a prime hyperideal of $R$.
}
\end{proof}




\end{document}